\documentclass{svmult}

\usepackage{amsfonts}
\usepackage[mathscr]{eucal}
\usepackage{amsmath}
\usepackage{amssymb}

\begin{document}

\title*{Algebraic Semigroups are Strongly $\pi$-regular}

\author{Michel Brion and Lex E. Renner}
\institute{Michel Brion 
\at Institut Fourier, Universit\'e de Grenoble, France\\
\email{Michel.Brion@ujf-grenoble.fr}
\and Lex E. Renner
\at Department of Mathematics, The University of Western Ontario, Canada\\
\email{lex@uwo.ca}}

\maketitle

%\setcounter{minitocdepth}{2}

%\dominitoc

\abstract{Let $S$ be an algebraic semigroup (not necessarily linear) 
defined over a field $F$. We show that there exists a positive integer 
$n$ such that $x^n$ belongs to a subgroup of $S(F)$ for any $x \in S(F)$. 
In particular, the semigroup $S(F)$ is strongly $\pi$-regular.
\keywords{algebraic semigroup, strong $\pi$-regularity.\\
MSC classes: 20M14, 20M32, 20M99.}}

\section{Introduction}   
\label{sec:intro}

A fundamental result of Putcha (see \cite[Thm.~3.18]{Put}) states that any 
{\em linear} algebraic semigroup $S$ over an algebraically closed field $k$
is strongly $\pi$-regular. The proof follows from the corresponding result 
for $M_n(k)$ (essentially the Fitting decomposition), combined with the fact 
that $S$ is isomorphic to a closed subsemigroup of $M_n(k)$, for some $n>0$. 
At the other extreme it is easy to see that any {\em complete} algebraic 
semigroup is strongly $\pi$-regular. It is therefore natural to ask
whether {\em any} algebraic semigroup $S$ is strongly $\pi$-regular. 
The purpose of this note is to provide an affirmative answer to this 
question, over an arbitrary field $F$; then the set $S(F)$ of points of 
$S$ over $F$ is an abstract semigroup (we shall freely use the terminology 
and results of \cite[Chap.~11]{Spr} for algebraic varieties defined over a field).

\section{The Main Results}   
\label{sec:result}

\begin{theorem}\label{thm:main}
Let $S$ be an algebraic semigroup defined over a subfield $F$ of $k$.
Then $S(F)$ is strongly $\pi$-regular, that is for any 
$x \in S(F)$, there exists a positive integer $n$ and an idempotent 
$e \in S(F)$ such that $x^n$ belongs to the unit group of $eS(F)e$.
\end{theorem}

\begin{proof}
\smartqed
We may replace $S$ with any closed subsemigroup 
defined over $F$ and containing some power of $x$. Denote by 
$\langle x \rangle$ the smallest closed subsemigroup of $S$ containing $x$,
that is, the closure of the subset $ \{ x^m, m >0 \}$; then 
$\langle x \rangle$ is defined over $F$ by \cite[Lem.~11.2.4]{Spr}. 
The subsemigroups $\langle x^n \rangle$, $n > 0$, form a family 
of closed subsets of $S$, and satisfy  
$\langle x^{mn} \rangle \subseteq 
\langle x^m \rangle \cap \langle x^n \rangle$. 
Thus, there exists a smallest such semigroup, say $\langle x^{n_0} \rangle$.
Replacing $x$ with $x^{n_0}$, we may assume that 
$S = \langle x \rangle = \langle x^n \rangle$ for all $n > 0$.

\begin{lemma}\label{lem:semi}
With the above notation and assumptions, $x S$ is dense in $S$. Moreover,
$S$ is irreducible.
\end{lemma}

\begin{proof}
\smartqed
Since $S = \langle x^2 \rangle$, the subset $ \{ x^n, n \geq 2 \}$
is dense in $S$. Hence $x S$ is dense in $S$ by an easy observation
(Lemma \ref{lem:dense}) that we will use repeatedly.

Let $S_1, \ldots,S_r$ be the irreducible components of $S$. 
Then each $x S_i$ is contained in some component $S_j$. Since $xS$ 
is dense in $S$, we see that $x S_i$ is dense in $S_j$. In particular, 
$j$ is unique and the map $\sigma : i \mapsto j$ is a permutation. By induction,
$x^n S_i$ is dense in $S_{\sigma^n(i)}$ for all $n$ and $i$; thus $x^n S_i$
is dense in $S_i$ for some $n$ and all $i$. Choose $i$ such that 
$x^n \in S_i$. Then it follows that $x^{mn} \in S_i$ for all $m$.
Thus, $\langle x^n \rangle \subseteq S_i$, and $S = S_i$ is irreducible.
\qed
\end{proof}

\begin{lemma}\label{lem:mono}
Let $S$ be an algebraic semigroup and let $x \in S$. 
Assume that $S = \langle x \rangle$ (in particular, $S$ is commutative), 
$xS$ is dense in $S$, and $S$ is irreducible. Then $S$ is a monoid 
and $x$ is invertible.
\end{lemma}

\begin{proof}
\smartqed
For $y \in S$, consider the decreasing sequence 
\[
\cdots \subseteq \overline{y^{n+1} S} \subseteq \overline{y^n S} 
\subseteq \cdots \subseteq \overline{yS }\subseteq S
\]
of closed, irreducible ideals of $S$. We claim that 
\[ \overline{y^d S} = \overline{y^{d+1} S}= \cdots, \] 
where $d := \dim(S) + 1$. Indeed, there exists $n \leq d$ such that
$\overline{y^{n+1} S} = \overline{y^n S}$, that is, 
$y^{n+1} S$ is dense in $\overline{y^n S}$. Multiplying by $y^{m-n}$
and using Lemma \ref{lem:dense}, it follows that $y^{m+1} S$ is dense 
in $\overline{y^m S}$ for all $m \geq n$ and hence for
$m \geq d$. This proves the claim.

We may thus set
\[
I_y := \overline{y^d S} = \overline{y^{d+1} S} = \cdots
\]
Then we have for all $y,z\in S$, 
\[
\overline{y^d I_z} = I_{yz} \subseteq I_z,
\]
since $y^d(z^d S) = (yz)^d S \subseteq z^d S$. Also, note that
$I_x = S$, and $I_e = e S$ for any idempotent $e$ of $S$. By
\cite[Sec.~2.3]{Br}, $S$ has a smallest idempotent $e_S$,
and $e_S S$ is the smallest ideal of $S$. In particular,
$e_S S \subseteq I_y$ for all $y$. Define 
\[
\mathscr{I} = \{ I \subseteq S\;|\; I = I_y\;\text{for some}\; y \in S\}.
\]
This is a set of closed, irreducible ideals, 
partially ordered by inclusion, with smallest element $e_SS$ 
and largest element $S$. If $S = e_S S$, then 
$S$ is a group and we are done. Otherwise, we may choose 
$I \in \mathscr{I}$ which covers $e_S S$ (since 
$\mathscr{I} \setminus \{ e_S S \}$ has minimal elements 
under inclusion). Consider
\[
T = \{ y\in S\:|\; yI \;\text{is dense in}\; I \}.
\]
If $y,z \in T$ then $\overline{yzI} = \overline{y\overline{zI}} = I$
and hence $T$ is a subsemigroup of $S$. Also, note that 
$T \cap e_S S = \emptyset$, since $e_S z I \subseteq e_S S$ is not dense
in $I$ for any $z \in S$. Furthermore $x \in T$. (Indeed, $x S$ is dense 
in $S$ and hence $x y^d S$ is dense in $\overline{y^dS}$ for all 
$y \in S$. Thus, $x \; \overline{y^d S}$ is dense in $\overline{y^d S}$;
in particular, $x I$ is dense in~$I$).

We now claim that 
\[
T = \{ y \in S\;|\; y^d I \not \subseteq e_S S \}.
\]
Indeed, if $y \in T$ then $y^d I$ is dense in $I$ and hence not contained
in $e_S S$. Conversely, assume that $y^d I \not \subseteq e_S S$ 
and let $z \in S$ such that $I = I_z$. Since
$\overline{y^d I} = \overline{y^d I_z} = I_{yz} \in \mathscr{I}$
and $\overline{y^d I} \subseteq I$, it follows that 
$\overline{y^d I} = I$ as $I$ covers $e_S S$.

By that claim, we have
\[ 
S\setminus T = \{ y \in S \;|\; y^d I \subseteq e_S S\}
= \{ y \in S \;|\; e_S y^d z = y^d z \; \text{for all} \; z \in I \}.
\]
Hence $S \setminus T$ is closed in $S$. Thus, $T$ is an open 
subsemigroup of $S$; in particular, $T$ is irreducible.
Moreover, since $x \in T$ and $xS$ is dense in $S$, it follows that
$x T$ is dense in $T$; also note that $\{ x^n, n >0 \}$ is dense in $T$.

Let $e_T \in T$ be the minimal idempotent, then $e_T \notin e_S S$ and hence 
the closed ideal $e_T S$ contains strictly $e_S S$. Since both are 
irreducible, we have 
$\dim(e_T T) = \dim(e_T S) > \dim(e_S S)$. Now the proof is completed 
by induction on $\kappa(S) := \dim(S) - \dim(e_S S)$.
Indeed, if $\kappa(S) = 0$, then $S = e_S S$ is a group. In the general
case, we have $\kappa(T) < \kappa(S)$. By the induction assumption, 
$T$ is a monoid and $x$ is invertible in $T$. As $T$ is dense in $S$, 
the neutral element of $T$ is also neutral for $S$, and hence $x$ 
is invertible in $S$.  
\qed
\end{proof}

By Lemmas \ref{lem:semi} and \ref{lem:mono}, there exists $n$ such
that $\langle x^n \rangle$ is a monoid defined over $F$, and $x^n$ is 
invertible in that monoid. To complete the proof of Theorem \ref{thm:main},
it suffices to show that the neutral element $e$ of $\langle x^n \rangle$
is defined over $F$. For this, consider the morphism
\[
\phi: S \times S \longrightarrow S, \quad (y,z) \longmapsto x^n y z.
\]
Then $\phi$ is the composition of the multiplication
\[
\mu : S \times S \longrightarrow S, \quad (y,z) \longmapsto yz
\]
and of the left multiplication by $x^n$; the latter is an
automorphism of $S$, defined over $F$. So $\phi$ is defined 
over $F$ as well, and the fiber $Z := \phi^{-1}(x^n)$ is isomorphic 
to $\mu^{-1}(e)$, hence to the unit group of $S$. In particular, 
$Z$ is smooth. Moreover, $Z$ contains $(e,e)$, and the tangent map 
\[ 
d\phi_{(e,e)} : T_{(e,e)} (S \times S) \longrightarrow T_{x^n} S 
\]
is surjective, since 
\[ 
d\mu_{(e,e)}: T_{(e,e)} (S \times S) = T_e S \times T_eS 
\longrightarrow T_e S 
\]
is just the addition. So $Z$ is defined over $F$ by 
\cite[Cor.~11.2.14]{Spr}. But $Z$ is sent to the point $e$ by $\mu$. 
Since that morphism is defined over $F$, so is $e$.
\qed
\end{proof} 

\begin{lemma}\label{lem:dense}
Let $X$ be a topological space, and $f : X \to X$ a continuous map. 
If $Y \subseteq X$ is a dense subset then $f(Y) \subseteq \overline{f(X)}$ 
is a dense subset.
\end{lemma}

\begin{proof}
\smartqed
Let $U \subseteq \overline{f(X)}$ be a nonempty open subset. 
Then $f^{-1}(U) \subseteq X$ is open, and nonempty since $f(X)$ is dense
in $\overline{f(X)}$. Hence $Y \cap f^{-1}(U) \neq \emptyset$. 
If $y \in Y\cap f^{-1}(U)$ then $f(y) \in f(Y)\cap U$. Hence 
$f(Y)\cap U \neq \emptyset$.
\qed
\end{proof}

\begin{remark}\label{rem:uni}
Given $x \in S$, there exists a {\em unique} idempotent $e = e(x) \in S$ 
such that $x^n$ belongs to the unit group of $e S e$ for some $n >0$. 
Indeed, we then have $x^n S x^n \subseteq e S e$; moreover, since
there exists $y \in eSe$ such that $x^n y = y x^n = e$, we also
have $e S e = x^n y S y x^n e \subseteq x^n S x^n$. Thus,
$x^n S x^n = e S e$. It follows that $x^{mn} S x^{mn}$ is a monoid 
with neutral element $e$ for any $m > 0$, which yields the desired
uniqueness.

In particular, if $x \in S(F)$ then the above idempotent $e(x)$ 
is an $F$-point of the closed subsemigroup $\langle x \rangle$. 
We now give some details on the structure of the latter 
semigroup. For $x,e,n$ as above, we have $x^n = ex^n = (ex)^n$, 
and $y(ex)^n = e$ for some $y \in H_e$ (the unit group of 
$e \, \langle x \rangle$). But then $e x \in H_e$ since $(y(ex)^{n-1})(ex) = e$. 
Thus, $ex^m = (ex)^m \in H_e$ for all $m > 0$. But if $m \geq n$ 
then $x^m = ex^m$. Thus, if $x \notin H_e$ then
there exists an unique $r > 0$ such that $x^r \notin H_e$ and $x^m \in H_e$ 
for any $m > r$. In particular, $x^r \in e \, \langle x \rangle$ for all 
$m \geq r$. Thus {\em we can write 
\[
\langle x \rangle = e \, \langle x \rangle \sqcup \{x,x^2,...,x^s \}
\]
for some $s < r$}. Notice also that {\em these $x^i$'s, with $i \leq s$, 
are all distinct} (if $x^i = x^j$ with $1 \leq i < j \leq s$, then 
$x^{i + s + 1 - j} = x^{s + 1} \in e \, \langle x \rangle$, a contradiction). 
Moreover, a similar decomposition holds for the semigroup of 
$F$-rational points.
  
The set $\{ ex^m, m > 0 \}$ is dense in $e \,\langle x \rangle$ 
by Lemma \ref{lem:dense}. But $ex^m = (ex)^m$, and $ex \in H_e$. 
So $e \, \langle x \rangle$ {\em is a unit-dense algebraic monoid}. 
Furthermore, if $\langle x^{m_0} \rangle$ is the 
smallest subsemigroup of $\langle x \rangle$ of the form 
$\langle x^m \rangle$, for some $m>0$, then $\langle x^{m_0} \rangle$ 
{\em is the neutral component of} $e \, \langle x \rangle$ 
(the unique irreducible component containing $e$). 
Indeed, $\langle x^{m_0} \rangle$ is irreducible by 
Lemma \ref{lem:semi}, and $y^{m_0} \in \langle x^{m_0} \rangle$ 
for any $y \in \langle x \rangle$ in view of Lemma \ref{lem:dense}. 
Thus, the unit group of $\langle x^{m_0} \rangle$ has finite index 
in the unit group of $\langle x \rangle$, and hence in that of 
$e \, \langle x \rangle$. 
\end{remark}

Finally, we show that Theorem \ref{thm:main} is self-improving
by obtaining the following stronger statement:  

\begin{corollary}\label{cor:gen}
Let $S$ be an algebraic semigroup. Then there exists $n > 0$ 
(depending only on $S$) such that $x^n \in H_{e(x)}$ for all $x \in S$, 
where $e : x \mapsto e(x)$ denotes the above map. Moreover, there 
exists a decomposition of $S$ into finitely many disjoint locally 
closed subsets $U_j$ such that the restriction of $e$ to each $U_j$ 
is a morphism.
\end{corollary}

\begin{proof}
\smartqed
We first show that for any irreducible 
subvariety $X$ of $S$, there exists a dense open subset $U$ of $X$
and a positive integer $n = n(U)$ such that $x^n \in H_{e(x)}$ for all 
$x \in U$, and $e\vert_U$ is a morphism. We will consider the semigroup
$S(k(X))$ of points of $S$ over the function field $k(X)$, and view 
any such point as a rational map from $X$ to $S$; the semigroup law on 
$S(k(X))$  is then given by pointwise multiplication of rational maps. 
In particular, the inclusion of $X$ in $S$ yields a point $\xi \in S(k(X))$ 
(the image of the generic point of $X$). By Theorem \ref{thm:main}, 
there exist a positive integer $n$ and points $e, y \in S(k(X))$
such that $e^2 = e$, $\xi^n e = e \xi^n = \xi^n$, $y e = e y = y$ and
$\xi^n y = y \xi^n = e$. Let $U$ be an open subset of $X$ on which 
both rational maps $e,y$ are defined. Then the above 
relations are equalities of morphisms $U \to S$, where $\xi$ is the 
inclusion. This yields the desired statements.

Next, start with an irreducible component $X_0$ of $S$ and let 
$U_0$ be an open subset of $X_0$ such that $e\vert_{U_0}$ is a 
morphism. Now let $X_1$ be an irreducible component of 
$X_0 \setminus U_0$ and iterate this construction. This yields
disjoint locally closed subsets $U_0,U_1,\ldots, U_j,\ldots$ such that 
$e\vert_{U_j}$ is a morphism for all $j$, and 
$X \setminus (U_0 \cup \cdots \cup U_j)$ is closed for all $j$. 
Hence $U_0 \cup \cdots \cup U_j = X$ for $j \gg 0$.
\qed
\end{proof}


\begin{thebibliography}{99}

\bibitem{Br} 
Brion, M.: 
On Algebraic Semigroups and Monoids,
Available on the arXiv:
\url{http://arxiv.org/pdf/1208.0675v4.pdf}


\bibitem{Put}  
Putcha, M.S.:
Linear Algebraic Monoids. 
Cambridge University Press, Cambridge (1988)

\bibitem{Spr} 
Springer, T.A.:
Linear Algebraic Groups. Second edition.
Birkh\"auser, Boston (1998)

\end{thebibliography}
\end{document}